\documentclass[12pt]{article}
\usepackage{amsthm,geometry,amssymb,amsmath,enumerate,float,tikz,cite}
\usetikzlibrary{decorations.pathreplacing}

\tikzstyle{every node}=[circle, draw, inner sep=0pt, minimum width=4pt]
\newtheorem{theorem}{Theorem}
\newtheorem{claim}{Claim}
\newtheorem{lemma}[theorem]{Lemma}
\newtheorem{conjecture}[theorem]{Conjecture}

\usepackage{setspace}

\begin{document}
\onehalfspace

\title{Uniquely restricted matchings in subcubic graphs}

\author{
Maximilian F\"{u}rst$^2$\and
Michael A. Henning$^{1,}$\thanks{Research
supported in part by the South African National Research Foundation and the
University of Johannesburg}
\and 
Dieter Rautenbach$^2$}

\date{}

\maketitle

\vspace{-1cm}

\begin{center}
${}^1$Department of Pure and Applied Mathematics\\
University of Johannesburg \\
Auckland Park, 2006 South Africa \\
{\small \tt Email:  mahenning@uj.ac.za} \\[3mm]

${}^2$Institute of Optimization and Operations Research \\
Ulm University, Ulm, Germany  \\
\small\tt Email:  \texttt{maximilian.fuerst, dieter.rautenbach}@uni-ulm.de
\end{center}

\begin{abstract}
A matching $M$ in a graph $G$ is uniquely restricted if no other matching in $G$ covers the same set of vertices.
We conjecture that every connected subcubic graph 
with $m$ edges and $b$ bridges
that is distinct from $K_{3,3}$  
has a uniquely restricted matching
of size at least $\frac{m+b}{6}$,
and we establish this bound with $b$ 
replaced by the number of bridges 
that lie on a path between two vertices of degree at most $2$. Moreover, we prove that every connected subcubic graph 
of order $n$ and girth at least $7$ 
has a uniquely restricted matching of size at least $\frac{n-1}{3}$,
which partially confirms a Conjecture of F\"{u}rst and Rautenbach 
(Some bounds on the uniquely restricted matching number, arXiv:1803.11032).
\end{abstract}
{\small \textbf{Keywords:} Matching; uniquely restricted matching; subcubic; bridge; girth

\newpage

\section{Introduction}
We consider only simple, finite, and undirected graphs, and use standard terminology.
A matching $M$ in a graph $G$ is {\it uniquely restricted} 
\cite{golumbic2001uniquely}
if no other matching in $G$ covers the same set of vertices, 
and $M$ is {\it acyclic} \cite{goddard2005generalized} if the subgraph induced by the set of vertices of $G$ covered by $M$
is a forest. 
The maximum sizes of a matching, 
a uniquely restricted matching, and 
an acyclic matching
are denoted by $\nu(G)$, $\nu_{ur}(G)$, and $\nu_{ac}(G)$, respectively. While unrestricted matchings are tractable \cite{lovasz2009matching},
uniquely restricted matchings and acyclic matchings are both NP-hard in general \cite{golumbic2001uniquely, goddard2005generalized}, 
and uniquely restricted matchings are also NP-hard in bipartite subcubic graphs \cite{mishra}.
This motivates the search for tight lower bounds. 
Golumbic, Hirst, and Lewenstein \cite{golumbic2001uniquely} 
observed that a matching $M$ in a graph $G$
is uniquely restricted if and only if there is no $M$-alternating cycle in  $G$, which implies $\nu_{ur}(G) \geq \nu_{ac}(G)$.
Hence, the main result in \cite{fura} implies the following.
\begin{theorem}\label{theoremfr}
If $G$ is a connected cubic graph with $m$ edges
that is distinct from $K_{3,3}$,
then $\nu_{ur}(G)\geq \frac{m}{6}$.
\end{theorem}
Since bridges lie in no cycles,
and, in particular, in no $M$-alternating cycles,
we believe that this result can be improved as follows.
\begin{conjecture}\label{conjecture1}
If $G$ is a connected subcubic graph 
with $m$ edges and $b$ bridges
that is distinct from $K_{3,3}$, then 
$\nu_{ur}(G)\geq \frac{m+b}{6}$.
\end{conjecture}
The bound in Conjecture \ref{conjecture1} is achieved with equality
for every subcubic graph $G$ that arises from a subcubic tree $T$
with matching number $\frac{n(T)-1}{3}$,
by replacing some of the vertices of degree $1$ in $T$
with endblocks isomorphic to $K_{2,3}$,
see Figure \ref{fig1}.
Note that there are infinitely many subcubic trees with matching number $\frac{n(T)-1}{3}$ \cite{henning2016tight}.
In fact, if we perform $k$ such replacements,
then $G$ has size $m=n(T)-1+6k$ and $b=n(T)-1$ bridges.
Since a uniquely restricted matching can contain at most one edge
from each $K_{2,3}$ subgraph,
it follows easily that 
$\nu_{ur}(G)=\frac{n(T)-1}{3}+k=\frac{m+b}{6}$.

\begin{figure}[H]
\centering\tiny
\begin{tikzpicture}[scale = 0.9] 
	    \node (A) at (1.5,3) {};
	    \node (B) at (0.5,2) {};
	    \node (C) at (1.5,2) {};
	    \node (F) at (2.5,2) {};
	    \node (b1) at (0,2) {};
	    \node (b2) at (1,2) {};
	    \node (b3) at (0.25,1) {};
	    \node (b4) at (0.75,1) {};
	    \node (f1) at (2,2) {};
	    \node (f2) at (3,2) {};
	    \node (f3) at (2.25,1) {};
	    \node (f4) at (2.75,1) {};

	    \foreach \from/\to in {A/C,b3/B,b3/b2,b4/b1,b4/B,b4/b2,f3/F,f3/f2,f4/f1,f4/F,f4/f2, A/F}
	    \draw [-] (\from) -- (\to);
	    
	    \foreach \from/\to in {b1/b3,A/B,f1/f3}
	    \draw [-,very thick] (\from) -- (\to);
	    
\end{tikzpicture}
\caption{A graph where Conjecture \ref{conjecture1} is tight.}\label{fig1}
\end{figure}
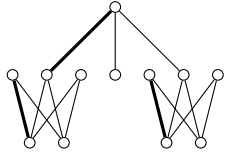

We prove the following weakening of Conjecture \ref{conjecture1}. 

A bridge in a graph is {\it good}
if it lies on a path between two vertices of degree at most $2$.

\begin{theorem}\label{theorem1}
If $G$ is a connected subcubic graph 
with $m$ edges and $b$ good bridges
that is distinct from $K_{3,3}$, then 
$\nu_{ur}(G)\geq \frac{m+b}{6}$.
\end{theorem}
Since every bridge in the graphs constructed above is good, 
Theorem \ref{theorem1} is also tight for these graphs.
F\"{u}rst and Rautenbach \cite{fura2} conjectured that 
$\nu_{ur}(G) \geq \frac{n-1}{3}$ 
for every connected subcubic graph $G$ of girth at least $5$.
We prove this conjecture for graphs of girth at least $7$.
\begin{theorem}\label{theorem2}
If $G$ is a connected subcubic graph of order $n$ and girth at least $7$, then $\nu_{ur}(G)\geq \frac{n-1}{3}$.
\end{theorem}
The next section contains the proofs of our two results.

\section{Proofs of Theorem \ref{theorem1} and Theorem \ref{theorem2}}

We immediately proceed to the proof of Theorem \ref{theorem1}.

\begin{proof}[Proof of Theorem \ref{theorem1}]
Suppose, for a contradiction, that $G$ is a counterexample 
of minimum size $m$.
Clearly, $G$ has order at least $2$.
Since no bridge in a cubic graph is good, 
Theorem \ref{theoremfr} implies that $G$ is not cubic.

\begin{claim}\label{claim1}
The minimum degree of $G$ is $2$.
\end{claim}
\begin{proof}[Proof of Claim \ref{claim1}:]
Suppose, for a contradiction, that $u$ is a vertex of degree $1$.
Let $v$ be the neighbor of $u$.
Let $G'=G-\{ u,v\}$
have $m'$ edges and $b'$ good bridges, see Figure \ref{fig1}.
Clearly, $m'\geq m-3$,
and $v$ is incident with at most $3$ good bridges.
Furthermore, 
since every vertex in $N_G(v)\setminus \{ u\}$ 
has degree less than $3$ in $G'$,
every good bridge of $G$ that belongs to $G'$
is also a good bridge of $G'$.
This implies $b'\geq b-3$.
Since adding $uv$ to a uniquely restricted matching in $G'$
yields a uniquely restricted matching in $G$,
the choice of $G$ implies the contradiction
$\nu_{ur}(G)\geq \nu_{ur}(G')+1\geq \frac{m'+b'}{6}+1\geq \frac{m+b}{6}$.
\end{proof}

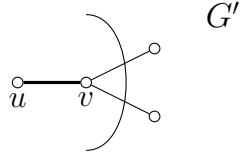
\begin{figure}[H] 
\centering\tiny
\begin{tikzpicture}[scale = 0.9] 
	    \node [label=below:\normalsize $u$](A) at (0,0.5) {};
	    \node [label=below:\normalsize $v$](B) at (1,0.5) {};
	    \node (C) at (2,0) {};
	    \node (D) at (2,1) {};

	    \foreach \from/\to in {B/C,B/D}
	    \draw [-] (\from) -- (\to);
	    
	    \draw [-,very thick] (A) -- (B);
	    \draw (1,-0.5) to[out=0,in=0] (1,1.5);
	    \pgftext[x=3cm,y=1.5cm] {\large $G^\prime$};
	   
\end{tikzpicture}
\caption{An illustration for Claim \ref{claim1}.} \label{fig1}
\end{figure}

\begin{claim}\label{claim2}
No triangle in $G$ contains two vertices of degree $2$.
\end{claim}
\begin{proof}[Proof of Claim \ref{claim2}:]
Suppose, for a contradiction, that $uvw$ is a triangle in $G$ 
such that $u$ and $v$ have degree $2$.
Let $G'=G-\{ u,v\}$
have $m'$ edges and $b'$ good bridges, see Figure \ref{fig2}.
Clearly, $m'\geq m-3$,
and neither $u$ nor $v$ is incident with a bridge.
Again, 
every good bridge of $G$ that belongs to $G'$
is also a good bridge of $G'$,
which implies $b'\geq b$.
Since adding $uv$ to a uniquely restricted matching in $G'$
yields a uniquely restricted matching in $G$,
the choice of $G$ implies the contradiction
$\nu_{ur}(G)\geq \nu_{ur}(G')+1\geq \frac{m'+b'}{6}+1>\frac{m+b}{6}$.
\end{proof}

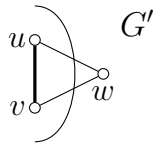
\begin{figure}[H] 
\centering\tiny
\begin{tikzpicture}[scale = 0.9] 
	    \node [label=left:\normalsize $v$](A) at (0,0) {};
	    \node [label=left:\normalsize $u$](B) at (0,1) {};
	    \node [label=below: \normalsize $w$](C) at (1,0.5) {};

	    \foreach \from/\to in {A/C,B/C}
	    \draw [-] (\from) -- (\to);
	    
	    \draw [-,very thick] (A) -- (B);
	    \draw (0,-0.5) to[out=0,in=0] (0,1.5);
	    \pgftext[x=1.5cm,y=1.25cm] {\large $G^\prime$};
\end{tikzpicture}
\caption{An illustration for Claim \ref{claim2}.} \label{fig2}
\end{figure}

\begin{claim}\label{claim3}
No two vertices of degree $2$ are adjacent in $G$.
\end{claim}
\begin{proof}[Proof of Claim \ref{claim3}:]
Suppose, for a contradiction, that $uv$ is an edge in $G$ 
such that $u$ and $v$ both have degree $2$.
Let $u'$ be the neighbor of $u$ distinct from $v$,
and let $N_G(u')=\{ u,w,w'\}$.

First, we assume that $uv$ is not a good bridge.
Let $G'=G-\{ u,v,u'\}$
have $m'$ edges and $b'$ good bridges, see the left of Figure \ref{fig3}.
Clearly, $m'\geq m-5$.
Since $u$ and $v$ have degree $2$, 
the edge incident with $v$ distinct from $uv$
as well as the edge $uu'$
are not good bridges.
If $u'w$ and $u'w'$ are both good bridges,
then, necessarily, also $uv$ would be a bridge,
and, in view of the degrees of $u$ and $v$,
the edge $uv$ would be a good bridge, which is a contradiction.
Therefore, $u'$ is incident with at most one good bridge.
As before, every good bridge of $G$ that belongs to $G'$
is also a good bridge of $G'$,
which implies $b'\geq b-1$.
Since adding $uv$ to a uniquely restricted matching in $G'$
yields a uniquely restricted matching in $G$,
the choice of $G$ implies the contradiction
$\nu_{ur}(G)\geq \nu_{ur}(G')+1\geq \frac{m'+b'}{6}+1\geq \frac{m+b}{6}$.

Hence, we may assume that $uv$ is a good bridge.
Let $G'=G-\{ u,v\}$
have $m'$ edges and $b'$ good bridges, see the right of Figure \ref{fig3}.
Clearly, $m'\geq m-3$.
As before, every good bridge of $G$ that belongs to $G'$
is also a good bridge of $G'$,
which implies $b'\geq b-3$.
Since adding $uv$ to a uniquely restricted matching in $G'$
yields a uniquely restricted matching in $G$,
the choice of $G$ implies the contradiction
$\nu_{ur}(G)\geq \nu_{ur}(G')+1\geq \frac{m'+b'}{6}+1\geq \frac{m+b}{6}$.
\end{proof}

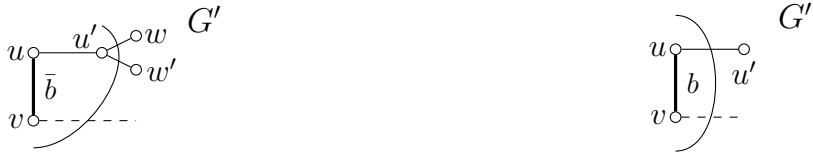
\begin{figure}[H]
\begin{minipage}{0.48\textwidth}
\centering\tiny
\begin{tikzpicture}[scale = 0.9] 
	    \node[label=left:\normalsize $u$] (A) at (0,1) {};
	    \node[label=left:\normalsize $v$] (B) at (0,0) {};
	    \node[label=above left:\normalsize $u^\prime$] (C) at (1,1) {};	    
	    \node[label=right:\normalsize $w$] (D) at (1.5,1.25) {};
	    \node[label=right:\normalsize $w^\prime$] (E) at (1.5,0.75) {};
	    \foreach \from/\to in {A/B,A/C,C/D,C/E}
	    \draw [-] (\from) -- (\to);
	    
	    \draw [-,very thick] (A) -- (B);
	    \draw [-,dashed] (B) -- (1.5,0);
	    \draw (0,-0.4) to[out=0,in=-30] (1,1.4);
	    \pgftext[x=2.5cm,y=1.5cm] {\large $G^\prime$};
	    \pgftext[x=0.25cm, y=0.5cm] {\normalsize $\bar{b}$};
\end{tikzpicture}
\end{minipage}
\begin{minipage}{0.48\textwidth}
\centering\tiny
\begin{tikzpicture}[scale=0.9]
	    \node[label=left:\normalsize $u$] (a) at (0,1) {};
	    \node[label=left:\normalsize $v$] (b) at (0,0) {};
	    \node[label=below:\normalsize $u^\prime$] (c) at (1,1) {};	    
	    
	    \foreach \from/\to in {a/b,a/c}
	    \draw [-] (\from) -- (\to);
	    
	    \draw [-,very thick] (a) -- (b);
	    \draw [-,dashed] (b) -- (1,0);
	    \draw (0,-0.5) to[out=0,in=0] (0,1.5);
	    \pgftext[x=1.75cm,y=1.5cm] {\large $G^\prime$};
	    \pgftext[x=0.25cm, y=0.5cm] {\normalsize $b$};
\end{tikzpicture}
\end{minipage}
\caption{An illustration for Claim \ref{claim3}. 
The label ``$b$'' indicates a good bridge,
while the label ``$\bar{b}$'' indicates an edge that is not a good bridge.} \label{fig3}
\end{figure}

Let $v$ be a vertex of degree $2$.
Let $u$ and $w$ be the neighbors of $v$.

\begin{claim} \label{claim4}
 $u$ and $w$ are not adjacent.
\end{claim}
\begin{proof}[Proof of Claim \ref{claim4}:]
Suppose, for a contradiction, that $u$ and $w$ are adjacent.
Clearly, both $u$ and $w$ are incident with at most one good bridge
and $v$ is incident with no  good bridge.

First, we assume that $w$ is incident with exactly one good bridge.
Let $G'=G-\{ u,v\}$
have $m'$ edges and $b'$ good bridges, see the left of Figure \ref{fig4}.
Clearly, $m'\geq m-4$.
As before, every good bridge of $G$ that belongs to $G'$
is also a good bridge of $G'$,
which implies $b'\geq b-1$.
Since adding $uv$ to a uniquely restricted matching in $G'$
yields a uniquely restricted matching in $G$,
the choice of $G$ implies the contradiction
$\nu_{ur}(G)\geq \nu_{ur}(G')+1\geq \frac{m'+b'}{6}+1> \frac{m+b}{6}$.

Hence, by symmetry between $u$ and $w$, we may assume that neither $u$ nor $w$ is incident with a good bridge.
Let $G'=G-\{ u,v,w\}$
have $m'$ edges and $b'$ good bridges, see the right of Figure \ref{fig4}.
Clearly, $m'\geq m-5$.
As before, every good bridge of $G$ that belongs to $G'$
is also a good bridge of $G'$,
which implies $b'\geq b$.
Since adding $uv$ to a uniquely restricted matching in $G'$
yields a uniquely restricted matching in $G$,
the choice of $G$ implies the contradiction
$\nu_{ur}(G)\geq \nu_{ur}(G')+1\geq \frac{m'+b'}{6}+1> \frac{m+b}{6}$.
\end{proof}

\begin{figure}[H]
\begin{minipage}{0.48\textwidth}
\centering\tiny
\begin{tikzpicture}[scale = 0.9] 
	    \node[label=left:\normalsize $v$] (v) at (0,0.5) {};
	    \node[label=above:\normalsize $u$] (u) at (1,1) {};
	    \node[label=below:\normalsize $w$] (w) at (1,0) {};	    
	    \foreach \from/\to in {u/w,v/w}
	    \draw [-] (\from) -- (\to);
	    
	    \draw [-,very thick] (u) -- (v);
	    \draw [-,dashed] (u) -- (2,1);
	    \draw [-,dashed] (w) -- (2,0);
	    \draw (-0.5,0) to[out=10,in=-90] (1.5,1.5);
	    \pgftext[x=2.5cm,y=1.5cm] {\large $G^\prime$};
	    \pgftext[x=1.5cm, y=0.25cm] {\normalsize $b$};
\end{tikzpicture}
\end{minipage}
\begin{minipage}{0.48\textwidth}
\centering\tiny
\begin{tikzpicture}[scale=0.9]
	    \node[label=left:\normalsize $v$] (v) at (0,0.5) {};
	    \node[label=above left:\normalsize $u$] (u) at (1,1) {};
	    \node[label=below left:\normalsize $w$] (w) at (1,0) {};	    
	    \foreach \from/\to in {u/w,v/w}
	    \draw [-] (\from) -- (\to);
	    
	    \draw [-,very thick] (u) -- (v);
	    \draw [-,dashed] (u) -- (2,1);
	    \draw [-,dashed] (w) -- (2,0);
	    \draw (0.7,-0.5) to[out=0,in=0] (0.7,1.5);
	    \pgftext[x=2.5cm,y=1.5cm] {\large $G^\prime$};
	    \pgftext[x=1.5cm, y=0.25cm] {\normalsize $\bar{b}$};
\end{tikzpicture}
\end{minipage}
\caption{An illustration for Claim \ref{claim4}.} \label{fig4}
\end{figure}
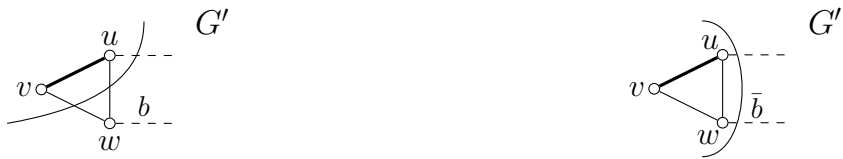

\begin{claim}\label{claim5}
$u$ and $w$ have at most two common neighbors.
\end{claim}
\begin{proof}[Proof of Claim \ref{claim5}:]
Suppose, for a contradiction, that 
$u$ and $w$ have three common neighbors.
Let $G'=G-\{ u,v,w\}$
have $m'$ edges and $b'$ good bridges.
Since 
$m'\geq m-6$, 
$b'\geq b$, and
adding $uv$ to a uniquely restricted matching in $G'$
yields a uniquely restricted matching in $G$,
the choice of $G$ implies the contradiction
$\nu_{ur}(G)\geq \nu_{ur}(G')+1\geq \frac{m'+b'}{6}+1\geq \frac{m+b}{6}$.
\end{proof}

\begin{claim}\label{claim6}
$v$ is the only common neighbor of $u$ and $w$.
\end{claim}
\begin{proof}[Proof of Claim \ref{claim6}:]
Suppose, for a contradiction, that 
$u$ and $w$ have two common neighbors.

First, we assume that $u$ is incident with a good bridge $uu'$.
Let $G'=G-\{ u,v,w,u'\}$
have $m'$ edges and $b'$ good bridges, see the left of Figure \ref{fig5}.
Since 
$m'\geq m-8$, 
$b'\geq b-4$, and
adding $uu'$ as well as $vw$ 
to a uniquely restricted matching in $G'$
yields a uniquely restricted matching in $G$,
the choice of $G$ implies the contradiction
$\nu_{ur}(G)\geq \nu_{ur}(G')+2\geq \frac{m'+b'}{6}+2\geq \frac{m+b}{6}$.

Hence, we may assume 
that neither $u$ nor $w$ is incident with a good bridge.
Let $G'=G-\{ u,v,w\}$
have $m'$ edges and $b'$ good bridges, see the right of Figure \ref{fig5}.
Since 
$m'\geq m-6$, 
$b'\geq b$, and
adding $vw$ 
to a uniquely restricted matching in $G'$
yields a uniquely restricted matching in $G$,
the choice of $G$ implies the contradiction
$\nu_{ur}(G)\geq \nu_{ur}(G')+1\geq \frac{m'+b'}{6}+1\geq \frac{m+b}{6}$.
\end{proof}
\begin{figure}[H]
\begin{minipage}{0.48\textwidth}
\centering\tiny
\begin{tikzpicture}[scale = 0.9] 
	    \node[label=left:\normalsize $v$] (v) at (0,0.5) {};
	    \node[label=above:\normalsize $u$] (u) at (1,1) {};
	    \node[label=below:\normalsize $w$] (w) at (1,0) {};
	    \node (a) at (2.5,0.5) {};
	    \node[label=above:\normalsize $u^\prime$] (u') at (2,1.5) {};
	    \foreach \from/\to in {u/v,u/a,w/a}
	    \draw [-] (\from) -- (\to);
	    
	    \draw [-,very thick] (v) -- (w);
	    \draw [-,very thick] (u) -- (u');
	    \draw [-,dashed] (u') -- (3,2);
	    \draw [-,dashed] (u') -- (3,1);
	    \draw [-,dashed] (w) -- (2,0);
	    \draw (1,-0.5) to[out=0,in=0] (2,2.25);
	    \pgftext[x=2.5cm,y=0cm] {\large $G^\prime$};
	    \pgftext[x=1.5cm, y=1.5cm] {\normalsize $b$};
\end{tikzpicture}
\end{minipage}
\begin{minipage}{0.48\textwidth}
\centering\tiny
\begin{tikzpicture}[scale=0.9]
    \node[label=left:\normalsize $v$] (v) at (0,0.5) {};
	    \node[label=above:\normalsize $u$] (u) at (1,1) {};
	    \node[label=below:\normalsize $w$] (w) at (1,0) {};
	    \node (a) at (2,0.5) {};
	
	    \foreach \from/\to in {u/v,u/a,w/a}
	    \draw [-] (\from) -- (\to);
	    
	    \draw [-,very thick] (v) -- (w);
	    \draw [-,dashed] (u) -- (2,1);
	    \draw [-,dashed] (w) -- (2,0);
	    \draw (1,-0.5) to[out=0,in=0] (1,1.5);
	    \pgftext[x=2.5cm,y=1cm] {\large $G^\prime$};

\end{tikzpicture}
\end{minipage}
\caption{An illustration for Claim \ref{claim6}.} \label{fig5}
\end{figure}
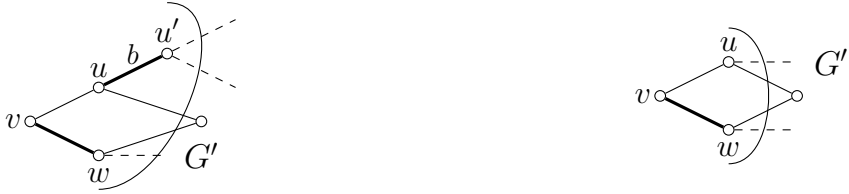
\begin{claim}\label{claim7}
At most one of the two edges incident with $v$ is a good bridge.
\end{claim}
\begin{proof}[Proof of Claim \ref{claim7}:]
Suppose, for a contradiction, that 
$uv$ and $vw$ are both good bridges.
Let $N_G(u)=\{ v,u',u''\}$.

First, we assume that $uu'$ and $uu''$ are both not good bridges.
Let $G'=G-\{ u,v\}$
have $m'$ edges and $b'$ good bridges, see the left of Figure \ref{fig6}.
Since 
$m'\geq m-4$, 
$b'\geq b-2$, and
adding $uv$ 
to a uniquely restricted matching in $G'$
yields a uniquely restricted matching in $G$,
the choice of $G$ implies the contradiction
$\nu_{ur}(G)\geq \nu_{ur}(G')+1\geq \frac{m'+b'}{6}+1\geq \frac{m+b}{6}$.

Hence, we may assume that $uu'$ is a good bridge.
Let $G'=G-\{ u,v\}+\{ u'w\}$
have $m'$ edges and $b'$ good bridges, see the right of Figure \ref{fig6}.
Clearly, $m'\geq m-3$.
Since $uu'$ and $vw$ are good bridges of $G$, 
the newly inserted edge $u'w$ is a good bridge of $G'$.
Note that this also implies
that every good bridge of $G$ that belongs to $G'$
is a good bridge of $G'$.
Since $u$ is incident with at most $3$ good bridges,
we obtain $b'\geq b-3$.
Let $M'$ be a uniquely restricted matching in $G'$.
If $u'w\not\in M'$, then let $M=M'\cup \{ uv\}$;
otherwise, let $M=(M'\setminus \{ u'w\})\cup \{ uu',vw\}$.
Since $M$ is a uniquely restricted matching in $G$,
the choice of $G$ implies the contradiction
$\nu_{ur}(G)\geq \nu_{ur}(G')+1\geq \frac{m'+b'}{6}+1\geq \frac{m+b}{6}$.
\end{proof}

\begin{figure}[H]
\begin{minipage}{0.48\textwidth}
\centering\tiny
\begin{tikzpicture}[scale = 0.9] 
	    \node[label=left:\normalsize $v$] (v) at (0,0.5) {};
	    \node[label=above:\normalsize $u$] (u) at (1,1) {};
	    \node[label=below:\normalsize $w$] (w) at (1,0) {};
	    \node[label=right:\normalsize $u^\prime$] (u') at (2,0.5) {};
	    \node[label=right:\normalsize $u^{\prime\prime}$] (u'') at (2,1.5) {};
	    \foreach \from/\to in {v/w,u/u',u/u''}
	    \draw [-] (\from) -- (\to);
	    
	    \draw [-,very thick] (u) -- (v);
	    \draw (-0.5,0.25) to[out=-30,in=-30] (1.2,1.5);
	    \pgftext[x=3cm,y=1.5cm] {\large $G^\prime$};
	    \pgftext[x=0.5cm, y=1cm] {\normalsize $b$};
	    \pgftext[x=0.5cm, y=0cm] {\normalsize $b$};
	    \pgftext[x=1.5cm, y=1.5cm] {\normalsize $\bar{b}$};
	    \pgftext[x=1.6cm, y=0.93cm] {\normalsize $\bar{b}$};
\end{tikzpicture}
\end{minipage}
\begin{minipage}{0.48\textwidth}
\centering\tiny
\begin{tikzpicture}[scale=0.9]
	    \node[label=left:\normalsize $v$] (v) at (0,0.5) {};
	    \node[label=above:\normalsize $u$] (u) at (1,1) {};
	    \node[label=below:\normalsize $w$] (w) at (1,0) {};
	    \node[label=right:\normalsize $u^\prime$] (u') at (2,0.5) {};
	    \node[label=right:\normalsize $u^{\prime\prime}$] (u'') at (2,1.5) {};
	    \foreach \from/\to in {v/w,u/u',u/u'',u/v}
	    \draw [-] (\from) -- (\to);
	    
	    \draw[-,dotted] (w) -- (u');
	    \draw (-0.5,0.25) to[out=-30,in=-30] (1.2,1.5);
	    \pgftext[x=3cm,y=1.5cm] {\large $G^\prime$};
	    \pgftext[x=0.5cm, y=1cm] {\normalsize $b$};
	    \pgftext[x=0.5cm, y=0cm] {\normalsize $b$};
	    \pgftext[x=1.6cm, y=0.9cm] {\normalsize $b$};

\end{tikzpicture}
\end{minipage}
\caption{An illustration for Claim \ref{claim7}.} \label{fig6}
\end{figure}
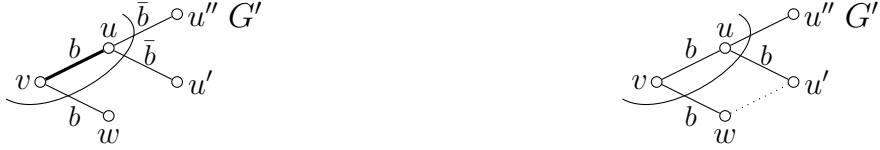

\begin{claim}\label{claim8}
No edge incident with $v$ is a good bridge.
\end{claim}
\begin{proof}[Proof of Claim \ref{claim8}:]
Suppose, for a contradiction, that 
$uv$ is a good bridge 
but $vw$ is not.

First, we assume that $u$ is incident with an edge 
that is not a good bridge.
Let $G'=G-\{ u,v\}$
have $m'$ edges and $b'$ good bridges, see the left of Figure \ref{fig7}.
Since 
$m'\geq m-4$, 
$b'\geq b-2$, and
adding $uv$ 
to a uniquely restricted matching in $G'$
yields a uniquely restricted matching in $G$,
the choice of $G$ implies the contradiction
$\nu_{ur}(G)\geq \nu_{ur}(G')+1\geq \frac{m'+b'}{6}+1\geq \frac{m+b}{6}$.

Hence, we may assume that 
all three edges incident with $u$ are good bridges.
For a neighbor $u'$ of $u$ distinct from $v$,
let the graph $G'=G-\{ u,v\}+\{ u'w\}$
have $m'$ edges and $b'$ good bridges, see the right of Figure \ref{fig7}.
Note that $u'w$ is not a good bridge of $G'$,
because, otherwise, $vw$ would be a good bridge of $G$.
Nevertheless, we obtain $m'\geq m-3$ and $b'\geq b-3$.
Let $M'$ be a uniquely restricted matching in $G'$.
If $u'w\not\in M'$, then let $M=M'\cup \{ uv\}$;
otherwise, let $M=(M'\setminus \{ u'w\})\cup \{ uu',vw\}$.
Since $M$ is a uniquely restricted matching in $G$,
the choice of $G$ implies the contradiction
$\nu_{ur}(G)\geq \nu_{ur}(G')+1\geq \frac{m'+b'}{6}+1\geq \frac{m+b}{6}$.
\end{proof}
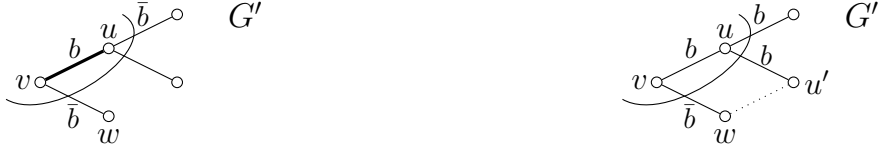
\begin{figure}[H]
\begin{minipage}{0.48\textwidth}
\centering\tiny
\begin{tikzpicture}[scale = 0.9] 
	    \node[label=left:\normalsize $v$] (v) at (0,0.5) {};
	    \node[label=above:\normalsize $u$] (u) at (1,1) {};
	    \node[label=below:\normalsize $w$] (w) at (1,0) {};
	    \node (u') at (2,0.5) {};
	    \node (u'') at (2,1.5) {};
	    \foreach \from/\to in {v/w,u/u',u/u''}
	    \draw [-] (\from) -- (\to);
	    
	    \draw [-,very thick] (u) -- (v);
	    \draw (-0.5,0.25) to[out=-30,in=-30] (1.2,1.5);
	    \pgftext[x=3cm,y=1.5cm] {\large $G^\prime$};
	    \pgftext[x=0.5cm, y=1cm] {\normalsize $b$};
	    \pgftext[x=0.48cm, y=0cm] {\normalsize $\bar{b}$};
	    \pgftext[x=1.5cm, y=1.5cm] {\normalsize $\bar{b}$};
\end{tikzpicture}
\end{minipage}
\begin{minipage}{0.48\textwidth}
\centering\tiny
\begin{tikzpicture}[scale=0.9]
	    \node[label=left:\normalsize $v$] (v) at (0,0.5) {};
	    \node[label=above:\normalsize $u$] (u) at (1,1) {};
	    \node[label=below:\normalsize $w$] (w) at (1,0) {};
	    \node[label=right:\normalsize $u^\prime$] (u') at (2,0.5) {};
	    \node (u'') at (2,1.5) {};
	    \foreach \from/\to in {v/w,u/u',u/u'',u/v}
	    \draw [-] (\from) -- (\to);
	    
	    \draw[-,dotted] (w) -- (u');

	    \draw (-0.5,0.25) to[out=-30,in=-30] (1.2,1.5);
	    \pgftext[x=3cm,y=1.5cm] {\large $G^\prime$};
	    \pgftext[x=0.5cm, y=1cm] {\normalsize $b$};
	    \pgftext[x=0.48cm, y=0cm] {\normalsize $\bar{b}$};
	   \pgftext[x=1.5cm, y=1.5cm] {\normalsize $b$};
	    \pgftext[x=1.6cm, y=0.9cm] {\normalsize $b$};

\end{tikzpicture}
\end{minipage}
\caption{An illustration for Claim \ref{claim8}.} \label{fig7}
\end{figure}

Now, we are in a position to derive the final contradiction.

First, we assume that $u$ and $w$ are both not incident with any good bridge.
Let $G'=G-\{ u,v,w\}$
have $m'$ edges and $b'$ good bridges, see the left of Figure \ref{fig8}.
Since 
$m'\geq m-6$, 
$b'\geq b$, and
adding $uv$ 
to a uniquely restricted matching in $G'$
yields a uniquely restricted matching in $G$,
the choice of $G$ implies the contradiction
$\nu_{ur}(G)\geq \nu_{ur}(G')+1\geq \frac{m'+b'}{6}+1\geq \frac{m+b}{6}$.

Next, we assume that $u$ is incident with two good bridges.
Let $G'=G-\{ u,v\}$
have $m'$ edges and $b'$ good bridges, see the middle of Figure \ref{fig8}.
Since 
$m'\geq m-4$, 
$b'\geq b-2$, and
adding $uv$ 
to a uniquely restricted matching in $G'$
yields a uniquely restricted matching in $G$,
the choice of $G$ implies the contradiction
$\nu_{ur}(G)\geq \nu_{ur}(G')+1\geq \frac{m'+b'}{6}+1\geq \frac{m+b}{6}$.

Hence, by symmetry between $u$ and $w$, 
we may assume that $u$ is incident 
with exactly one good bridge $uu'$,
and that $w$ is incident with at most one good bridge.
Let $G'=G-\{ u,v,w,u'\}$
have $m'$ edges and $b'$ good bridges, see the right of Figure \ref{fig8}.
Since 
$m'\geq m-8$, 
$b'\geq b-4$, and
adding $uu'$ as well as $vw$ 
to a uniquely restricted matching in $G'$
yields a uniquely restricted matching in $G$,
the choice of $G$ implies the contradiction
$\nu_{ur}(G)\geq \nu_{ur}(G')+2\geq \frac{m'+b'}{6}+2\geq \frac{m+b}{6}$,
which completes the proof.

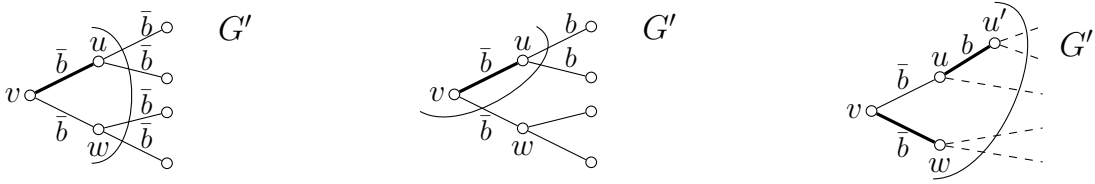
\begin{figure}[H]
\begin{minipage}{0.33\textwidth}
\centering\tiny
\begin{tikzpicture}[scale = 0.9] 
	    \node[label=left:\normalsize $v$] (v) at (0,0.5) {};
	    \node[label=above:\normalsize $u$] (u) at (1,1) {};
	    \node[label=below:\normalsize $w$] (w) at (1,0) {};
	    \node (u') at (2,0.75) {};
	    \node (u'') at (2,1.5) {};
	    \node (w') at (2,0.25) {};
	    \node (w'') at (2,-0.5) {};
	    \foreach \from/\to in {v/w,u/u',u/u'', w/w', w/w''}
	    \draw [-] (\from) -- (\to);
	    
	    \draw [-,very thick] (u) -- (v);
	    \draw (0.9,-0.5) to[out=0,in=0] (0.9,1.5);
	    \pgftext[x=3cm,y=1.5cm] {\large $G^\prime$};
	    \pgftext[x=0.46cm, y=1cm] {\normalsize $\bar{b}$};
	    \pgftext[x=0.46cm, y=0cm] {\normalsize $\bar{b}$};
	    \pgftext[x=1.7cm, y=1.55cm] {\normalsize $\bar{b}$};
	     \pgftext[x=1.7cm, y=1.05cm] {\normalsize $\bar{b}$};
	      \pgftext[x=1.7cm, y=0.4cm] {\normalsize $\bar{b}$};
	       \pgftext[x=1.7cm, y=-0.1cm] {\normalsize $\bar{b}$};
\end{tikzpicture}
\end{minipage}
\begin{minipage}{0.32\textwidth}
\centering\tiny
\begin{tikzpicture}[scale=0.9]
	    \node[label=left:\normalsize $v$] (v) at (0,0.5) {};
	    \node[label=above:\normalsize $u$] (u) at (1,1) {};
	    \node[label=below:\normalsize $w$] (w) at (1,0) {};
	    \node (u') at (2,0.75) {};
	    \node (u'') at (2,1.5) {};
	    \node (w') at (2,0.25) {};
	    \node (w'') at (2,-0.5) {};
	    \foreach \from/\to in {v/w,u/u',u/u'', w/w', w/w''}
	    \draw [-] (\from) -- (\to);
	    
	    \draw [-,very thick] (u) -- (v);
	    \draw (-0.5,0.25) to[out=-30,in=-30] (1.2,1.5);
	    \pgftext[x=3cm,y=1.5cm] {\large $G^\prime$};
	    \pgftext[x=0.46cm, y=1cm] {\normalsize $\bar{b}$};
	    \pgftext[x=0.46cm, y=0cm] {\normalsize $\bar{b}$};
	    \pgftext[x=1.7cm, y=1.55cm] {\normalsize $b$};
	     \pgftext[x=1.7cm, y=1.05cm] {\normalsize $b$};

\end{tikzpicture}
\end{minipage}
\begin{minipage}{0.33\textwidth}
\centering\tiny
\begin{tikzpicture}[scale = 0.9] 
	    \node[label=left:\normalsize $v$] (v) at (0,0.5) {};
	    \node[label=above:\normalsize $u$] (u) at (1,1) {};
	    \node[label=below:\normalsize $w$] (w) at (1,0) {};
	    \node[label=above:\normalsize $u^\prime$] (u'') at (1.8,1.5) {};
	    \foreach \from/\to in {v/u}
	    \draw [-] (\from) -- (\to);
	    
	    \draw [-,very thick] (v) -- (w);
	    \draw [-,very thick] (u) -- (u'');
	    
	    \draw [-,dashed] (u) -- (2.5,0.75);
	    \draw [-,dashed] (w) -- (2.5,0.25);
	    \draw [-,dashed] (w) -- (2.5,-0.25);
	     \draw [-,dashed] (u'') -- (2.5,1.75);
	      \draw [-,dashed] (u'') -- (2.5,1.25);
	    \draw (0.9,-0.5) to[out=0,in=0] (1.8,2.1);
	    \pgftext[x=3cm,y=1.5cm] {\large $G^\prime$};
	    \pgftext[x=0.46cm, y=1cm] {\normalsize $\bar{b}$};
	    \pgftext[x=0.46cm, y=0cm] {\normalsize $\bar{b}$};
	    \pgftext[x=1.4cm, y=1.48cm] {\normalsize $b$};
\end{tikzpicture}
\end{minipage}
\caption{An illustration of the final contradiction.} \label{fig8}
\end{figure}
\end{proof}
In order to prove Theorem \ref{theorem2}, 
we need the following lemma.

\begin{lemma}\label{lemma1}
If $G$ is a connected subcubic graph of order $n$ and girth at least $7$
that is not a tree and not cubic, then $\nu_{ur}(G)\geq \frac{n}{3}$.
\end{lemma}
\begin{proof}
Suppose, for a contradiction, that $G$ is a counterexample of minimum order. 

First, we assume that $G$ has a vertex $u$ of degree $1$.
Let $v$ be the unique neighbor of $u$, 
and let $G'=G-\{ u,v\}$.
Note that $G'$ has at most $2$ components,
none of which is cubic.
Since $G$ is not a tree, 
at most one component of $G'$ is a tree,
and such a component $K$ has a uniquely restricted matching of size at least
$\frac{n(K)-1}{3}$.
Therefore, 
since adding $uv$ to a uniquely restricted matching in $G^\prime$ 
yields a uniquely restricted matching in $G$,
we obtain the contradiction
$\nu_{ur}(G)\geq \nu_{ur}(G')+1\geq \frac{n-2-1}{3}+1=\frac{n}{3}$.
Hence, we may assume that $G$ has minimum degree $2$.

Let $P:u_1v_1u_2v_2\ldots u_kv_ku_{k+1}$
be a maximal path in $G$ 
such that the vertices $v_1,\ldots,v_k$ all have degree $2$ in $G$.
Let $G'=G-V(P)$,
let ${\cal T}$ be the set of components of $G'$ that are trees,
and let $c=|{\cal T}|$, see Figure \ref{fig9}.
If $T$ is in ${\cal T}$, 
then the minimum degree of $G$ implies
that there are at least two edges between $V(P)$ and $V(T)$.
Since there are at most $k+3$ edges between $V(P)$ and $V(G')$,
we obtain $c\leq \frac{k+3}{2}$.
If $c\leq k-1$, then, 
since adding $u_1v_1,\ldots,u_kv_k$ 
to a uniquely restricted matching in $G^\prime$ 
yields a uniquely restricted matching in $G$,
we obtain the contradiction
$\nu_{ur}(G)\geq \nu_{ur}(G')+k
\geq \frac{n-n(P)-c}{3}+k
\geq \frac{n-(2k+1)-(k-1)}{3}+k
=\frac{n}{3}$.
Hence, we may assume that $c\geq k$,
which, together with $c\leq \frac{k+3}{2}$, 
implies that $k\leq 3$.

Let $E$ be the set of edges of $G$ between $V(P)$
and a component in ${\cal T}$.
If neither $u_1$ nor $u_{k+1}$ are incident with an edge in $E$,
then $c\leq \frac{k-1}{2}$, contradicting $c\geq k$.
Hence, by symmetry, 
we may assume that $u_1w$ belongs to $E$.
Let $T$ be the component of $G'$ that contains $w$.

If $k\leq 2$, then, by the girth condition,
$u_1w$ is the only edge in $E$ incident with $w$.
By the maximality of $P$,
it follows that $w$ has degree $3$ in $G$.
This implies that $T$ has two endvertices $x$ and $y$.
Since $k\leq 2$, we may assume, by symmetry,
that $x$ is adjacent to $u_1$.
Again using the girth condition,
we obtain that $x$ is incident with exactly one edge in $E$.
This implies that $x$ has degree $2$ in $G$,
and, if $z$ is the neighbor of $x$ in $T$,
then the path $zxu_1v_1\ldots u_kv_ku_{k+1}$
contradicts the maximality of $P$.
Hence, we may assume that $k=3$.

Since $E$ contains at most $6$ edges,
$c=3$, and every component in ${\cal T}$
is incident with at least two edges in $E$,
all edges of $G$ that are incident with a vertex of $P$
and do not belong to $P$, belong to $E$,
and between $V(P)$ and every tree in ${\cal T}$
there are exactly two edges.

Let $u_2w'$ be in $E$,
and let $T'$ be the component of $G'$ that contains $w'$.
By the girth condition, 
$u_2w'$ is the only edge in $E$ incident with $w'$.
This implies that $T'$ has an endvertex $x'$ distinct from $w'$.
Since there are exactly two edges between $V(P)$ and $V(T')$,
the maximality of $P$ implies that $x'$ is adjacent to $u_3$.
If the two trees in ${\cal T}\setminus \{ T'\}$ are isolated vertices,
then $G$ contains a cycle of length $4$, which is a contradiction.
Hence, ${\cal T}\setminus \{ T'\}$ contains a tree $T''$ that has 
at least two endvertices $w''$ and $x''$.
By symmetry, we may assume that $x''$ is adjacent to $u_1$.
Since $x''$ is incident with only one edge in $E$,
it has degree $2$ in $G$,
and, if $z''$ is the neighbor of $x''$ in $T''$,
then the path $z''x''u_1v_1\ldots u_kv_ku_{k+1}$
contradicts the maximality of $P$.
\end{proof}


\begin{figure}[H]
\centering\tiny
\begin{tikzpicture}[scale = 0.9] 

	    \foreach \from/\name in {0/A,1/B,2/C,3/D,6/E,7/F,8/G}
	    \node (\name) at (\from,0) {};
	    
	    \draw[-,dotted] (4,0) -- (5,0);
	    \foreach \from/\to in {B/C, F/G}
	    \draw [-] (\from) -- (\to);
	    
	    \foreach \from/\to in {A/B,E/F, C/D}
	    \draw [-,very thick] (\from) -- (\to);
	    
	    \foreach \from/\x in {C/2,E/6}
	    \draw [-, dotted] (\from) -- (\x,1);
	   
	   \draw[-,dotted] (A) -- (-0.5,1);
	   \draw[-,dotted] (A) -- (0.5,1);
	   \draw[-,dotted] (G) -- (7.5,1);
	   \draw[-,dotted] (G) -- (8.5,1);
	   
	   \draw (-0.5,0) to[out=15,in=165] (8.5,0);
	   \draw (3.25,2) to[out=-30,in=-150] (4.25,2);
	   
	   \draw (-0.25,3) ellipse (1 and 0.5); 
	   \draw[-,dotted] (1.25,3) -- (2.25,3);
	   \draw (3.75,3) ellipse (1cm and 0.5cm);
	   \draw (7.75,3) ellipse (2cm and 1cm);
	   
	   \draw[-,dotted] (-0.25,2.5) -- (-0.5,1.5);
	   \draw[-,dotted] (-0.25,2.5) -- (0,1.5);
	   \draw[-,dotted] (3.75,2.5) -- (3.5,1.5);
	   \draw[-,dotted] (3.75,2.5) -- (4,1.5);
	   
	   \draw[thick,decorate,decoration={brace,amplitude=8pt}] (-1.25,3.5) -- (4.75,3.5);
	   \pgftext[x=1.75cm,y=4.2cm] {\large $\cal{T}$};
	   \pgftext[x=-1cm,y=0cm] {\large $P$};
	   \pgftext[x=4.75cm,y=2cm] {\normalsize $\geq 2$};
	   \pgftext[x=7.75cm,y=3cm] {\normalsize $G-(V(P)\cup V(\cal{T}))$};
	    
	    \foreach \from/\name in {0/u_1,1/v_1,2/u_2,3/v_2,6/u_k,7/v_k,8/u_{k+1}}
	    \pgftext[x=\from cm,y=-0.35cm] {\normalsize $\name$};
\end{tikzpicture}
\caption{An illustration of Lemma \ref{lemma1}.} \label{fig9}
\end{figure}
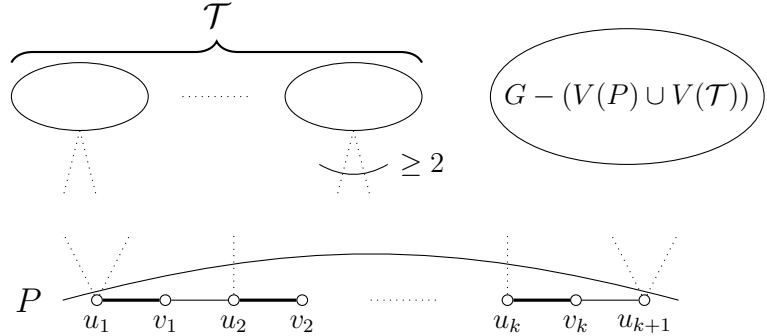

It is now straightforward to prove Theorem \ref{theorem2}.

\begin{proof}[Proof of Theorem \ref{theorem2}]
Suppose, for a contradiction, that $G$ is a counterexample of minimum order. 

First, we assume that $G$ has a vertex $u$ of degree $1$.
Let $v$ be the unique neighbor of $u$, 
and let $G'=G-\{ u,v\}$.
Since $G$ has order $n-2$ and at most $2$ components,
and adding $uv$ to a uniquely restricted matching in $G^\prime$ 
yields a uniquely restricted matching in $G$,
we obtain the contradiction
$\nu_{ur}(G)\geq \nu_{ur}(G')+1\geq \frac{n-2-2}{3}+1=\frac{n-1}{3}$.
Hence, we may assume that $G$ has minimum degree $2$.
By Lemma \ref{lemma1},
we may assume that $G$ is cubic.
Let $u$ be an endvertex of some spanning tree of $G$,
and let $G'=G-u$.
Clearly, $G'$ is connected, subcubic and not cubic, and it is not a tree.
Since every uniquely restricted matching in $G'$
is a uniquely restricted matching in $G$,
Lemma \ref{lemma1} implies
$\nu_{ur}(G)\geq \nu_{ur}(G')\geq \frac{n-1}{3}$,
which completes the proof. 
\end{proof}

\end{document}